\documentclass[11pt,reqno]{amsart}
\usepackage{latexsym}

\RequirePackage{amsmath}   \RequirePackage{amssymb}
\RequirePackage{amsthm}   \RequirePackage{epsfig}
\theoremstyle{plain}
\newtheorem{theorem}{Theorem}
\newtheorem{proposition}[theorem]{Proposition}

\newtheorem{lemma}[theorem]{Lemma}
\newtheorem*{corollary*}{Corollary}
\newtheorem{theoremO}{Theorem}
\newtheorem{lemmO}[theoremO]{Lemma}

\theoremstyle{definition}
\newtheorem{definition}{Definition}

\theoremstyle{remark}
\newtheorem{remark}{Remark}
\newtheorem{example}[remark]{Example}

\newcommand{\SC}{{\mathbb C}}  \newcommand{\D}{{\mathbb D}}  
  
\newcommand{\al}{\alpha}    
  \newcommand{\ve}{\varepsilon}  \newcommand{\ze}{\zeta}
    
  \newcommand{\vp}{\varphi}  \newcommand{\om}{\omega}
  
\newcommand{\cB}{{\mathcal B}}

\newcommand{\be}{\begin{equation}}
\newcommand{\ee}{\end{equation}}
\newcommand{\bea}{\begin{eqnarray}}
\newcommand{\eea}{\end{eqnarray}}

\begin{document}
\title{On harmonic Bloch-type mappings}
\author{I. Efraimidis \and J. Gaona \and R. Hern\'andez \and O. Venegas }
\thanks{The first author is supported by a fellowship of the International Program of Excellence in Mathematics at Universidad Aut\'onoma de Madrid (422Q101) and also partially supported by MINECO grant MTM2015-65792-P (ERDF/FEDER). The second and third author are partially supported by Fondecyt Grants \# 1150284.}

\address{Departamento de Matem\'aticas, Universidad Aut\'onoma de Madrid, 28049 Madrid, Spain.} \email{iason.efraimidis@uam.es}
\address{Instituto de Matem\'aticas. Facultad de Ciencias, Universidad de Valpara\'iso, Valpar\'iso, Chile.}\email{jhonattangaona@hotmail.com}
\address{Facultad de Ciencias y Tecnolog\'ia, Universidad Adolfo Ib\'a\~nez, Vi\~na del Mar, Chile.} \email{rodrigo.hernandez@uai.cl}
\address{Departamento de Ciencias Matem\'{a}ticas y F\'{\i}sicas. Facultad de Ingenier\'{\i}a, Universidad Cat\'olica de Temuco, Temuco, Chile.} \email{ovenegas@uct.cl}

\maketitle

\begin{abstract}
Let $f$ be a complex-valued harmonic mapping defined in the unit disk $\D$. We introduce the following notion: we say that $f$ is a Bloch-type function if its Jacobian satisfies 
$$
\sup_{z\in\D}(1-|z|^2)\sqrt{|J_f(z)|}<\infty.
$$
This gives rise to a new class of functions which generalizes and contains the well-known analytic Bloch space. We give estimates for the schlicht radius, the growth and the coefficients of functions in this class. We establish an analogue of the theorem which states that an analytic $\vp$ is Bloch if and only if there exists $c>0$ and a univalent $\psi$ such that $\vp = c \log \psi
'$.
\end{abstract}

\section{Introduction}

\subsection{Bloch functions} Let $\D$ be the unit disk in the complex plane and $\vp$ an analytic function defined in $\D$. We say that $\vp$ is a \emph{Bloch function} if
\be \label{Bloch}
\beta(\vp) = \sup_{z\in\D}(1-|z|^2)|\varphi'(z)|<\infty.
\ee
This defines a seminorm, and the Banach space $\cB$ of all Bloch functions equipped with the norm $\|\varphi\|_\mathcal B = |\varphi(0)|+\beta(\vp) $ is called \emph{Bloch space}. We refer to \cite{ACP}, \cite{Da}, \cite{Po70}, \cite{Pom}, \cite{SW} and \cite{Zhu} for information on the Bloch space.

For $\vp$ analytic in $\D$ the \emph{schlicht radius} $d_\vp(z)$ is defined as the radius of the largest disk lying on the Riemann image of $\vp$ and centered at the point $\vp(z)$, whenever $z$ is not a branch point, \emph{i.e.}$\!$ if $\vp'(z)\neq 0$. At a branch point of $\vp$ the schlicht radius is defined as zero. It was shown in \cite{SW} that every analytic function satisfies
$$
d_\vp(z) \, \leq \, (1-|z|^2)|\vp'(z)|, \qquad  z\in\D.
$$
A similar inequality in the reverse direction was also shown in \cite[\S31]{SW} for the case when the schlicht radius is uniformly bounded. Thus, $\vp\in\cB$ if and only if $\sup_{z\in\D}d_\vp(z)<\infty$.

For univalent functions these inequalities take the simpler form of
\be \label{rad}
\frac{1}{4}(1-|z|^2)|\vp'(z)| \, \leq \, d_\vp(z) \, \leq \, (1-|z|^2)|\vp'(z)|, \qquad  z\in\D
\ee
due to Koebe's $1/4$-Theorem. Note that in this case $d_\vp(z)$ is simply the distance between $\vp(z)$ and the boundary of $\vp(\D)$ and, therefore, $\vp\in\cB$ if and only if $\vp(\D)$ does not contain arbitrarily large disks. 

Yet another close connection between Bloch functions and univalent functions was found in \cite{Po70}. Namely, if $\psi$ is univalent then $\beta(\log\psi') \leq 6$ and, conversely, if $\beta(\vp)\leq 1$ then $\vp=\log \psi'$ for some univalent function $\psi$. 

\subsection{Harmonic mappings} A planar harmonic mapping is a complex-valued harmonic function $f$ defined on a domain $\Omega\subset\mathbb{C}$. When $\Omega$ is simply connected, the mapping has a \emph{canonical decomposition} $f=h+\overline{g}$, where $h$ and $g$ are analytic in $\Omega$. Since the Jacobian of $f$ is given by $J_f=|h'|^2 - |g'|^2$, it is locally univalent and \emph{orientation-preserving} if and only if $|g'|< |h'|$, or equivalently, if $h'(z)\neq0$ and the dilatation $\omega=g'/h'$ has the property $|\omega(z)|<1$ in $\Omega$. We say that $f$ is \emph{orientation-reversing} if $\overline{f}$ is orientation-preserving.

Since the mid-80s and especially after the work of J. Clunie and T. Sheil-Small \cite {CS} in 1984, there has been a great interest in trying to extend the classic results of the analytic world to their harmonic analogues. Some work in this direction for the Bloch space was done by F. Colonna \cite{Co}, whose point of departure was the \emph{metric} characterization of $\cB$, namely, $f\in\cB$ if and only if $f$ is Lipschitz between $\D$ endowed with the hyperbolic metric and $\SC$ endowed with the euclidean metric. For a harmonic mapping $f=h+\overline g\,$ this Lipschitz condition was proved in \cite{Co} to be equivalent to both $h$ and $g$ belonging to $\cB$.

The \emph{schlicht radius} $d_f(z)$ of a harmonic mapping $f=h+\overline{g}$ is defined as the radius of the largest disk which is the injective image of some subdomain of $\D$ and is centered at $f(z)$. We set $d_f(z)=0$ if no such disk exists. A generalization of the \emph{geometric} definition of Bloch functions would be to ask that $f$ satisfy $\sup_{z\in\D}d_f(z)<\infty$. However, we shall prove in Lemma \ref{radius} that if $f$ is univalent and normalized then
$$
\frac{1}{16}(1-|z|^2) (|h'(z)|-|g'(z)|) \; \leq \; d_f(z) \; \leq \; \frac{\pi}{2} (1-|z|^2) |h'(z)|, \qquad z\in\D.
$$
It can also be shown that no two of the above three quantities are comparable. Therefore, an analytic characterization of the geometric definition for harmonic mappings is, as far as we know, yet to be found. 

\subsection{Harmonic Bloch-type functions} Our starting point will be the \emph{analytic} definition \eqref{Bloch}. Noting that the Jacobian of an analytic function $\vp$ is given by $J_\vp = |\vp'|^2$, we feel justified in introducing the following definition. 

\begin{definition}
Let $f=h+\overline{g}$ be harmonic in $\D$. We say that $f$ is a \emph{Bloch-type function} if
$$
\beta(f) \; = \; \sup_{z\in\D}(1-|z|^2)\sqrt{|J_f(z)|}<\infty.
$$
We denote this class of functions by $\cB_H$.
\end{definition}

Indeed, we shall see in Section 2 that this definition gives rise to a class rather than a linear space. However, $\cB_H$ contains the Bloch space defined in \cite{Co}. We shall prove that $\cB_H$ is both affine and linearly invariant. In Section 3 we show a connection between $\cB_H$ and univalent harmonic mappings that resembles Pommerenke's theorem \cite{Po70}. We also study the schlicht radius in $\cB_H$. In Section 4 we give growth and coefficients estimates for sense-preserving functions in $\cB_H$.

\section{The class of harmonic Bloch-type mappings}

Our first task will be to show the affine and linear invariance of $\cB_H$. Throughout the paper we will denote by $\vp_\al \, (\al\in\D)$ the disk automorphism given by $\vp_\al(z)=(\al+z)/(1+\overline{\al}z), z\in\D$. 

\begin{proposition} If $f\in\mathcal{B}_H$ then 
\begin{enumerate}
\item[(i)] $a f+b\overline{f} \in\cB_H$ for any $a,b\in\mathbb C$. (affine invariance) \\

\item[(ii)] $f\circ \vp_\al\in\cB_H$ for any $\al\in\D$.  (linear invariance) 
\end{enumerate}
\end{proposition}

\begin{proof} Let $f=h+\overline{g}$. To prove (i) we write
$$
F = a f+b\, \overline{f} = a h +b g +\overline{\overline{a}g+\overline{b}h}
$$
and compute
$$
J_F = |a h' +b g'|^2 - |\overline{a}g'+\overline{b}h'|^2 = (|a|^2-|b|^2) \, J_f.
$$
The assertion now easily follows. 

For claim (ii) we write $F=f\circ \vp_\al = H+\overline{G}$ and compute
$$
H'(z) = \frac{h'\big(\vp_\al(z)\big) (1-|\al|^2) }{(1+\overline{\al}z)^2}, \qquad G'(z) = \frac{g'\big(\vp_\al(z)\big) (1-|\al|^2) }{(1+\overline{\al}z)^2}. 
$$
Hence
\begin{align*}
(1-|z|^2)\sqrt{|J_F(z)|} &= \frac{(1-|z|^2) (1-|\al|^2) }{|1+\overline{\al}z|^2}\sqrt{|J_f\big(\vp_\al(z)\big)|} \\ 
& = \big(1-|\vp_\al(z)|^2\big) \sqrt{|J_f\big(\vp_\al(z)\big)|}.
\end{align*}
Taking the supremum over $z\in\D$ we get that $\beta(F) = \beta(f)$.
\end{proof}

In what follows, Example \ref{ex-no-space} shows that $\cB_H$ is not a linear space. It also shows that functions in $\cB_H$ may grow arbitrarily fast. Hence, in order to get growth and coefficient estimates in Section 4 we shall restrict ourselves to sense-preserving functions in $\cB_H$. 

\begin{example} \label{ex-no-space}
Consider an analytic function $h$ for which $h'(z) = (1-z)^{-p}$, for some $p>2$. Set $f = h +\overline{h} = 2 \,{\rm Re}\, \{h\}$ and see that, since $J_f \equiv 0$, $f$ belongs to $\cB_H$. Obviously, the identity $\text{id}(z)= z$ belongs to $\cB_H$, but we will see that $f+\text{id}$ does not. Indeed,
$$
J_{f+\text{id}} = |h'+1|^2 - |h'|^2 = 1 +2{\rm Re}\, \{h'\}
$$ 
and therefore, for $0<x<1$ we have
$$
(1-x^2)^2 |J_{f+\text{id}}(x)| = (1+x)^2 \frac{2+(1-x)^p}{(1-x)^{p-2}} \longrightarrow \infty
$$
as $x\to1^-$. 
\end{example} 

Example \ref{ex-extr} shows that the harmonic Bloch space considered in \cite{Co} is strictly contained in $\cB_H$. Recall that in \cite{Co} the definition of a Bloch function $f= h + \overline{g}$ is equivalent to both $h$ and $g$ belonging to $\cB$. 

\begin{example} \label{ex-extr}
Let  $f = h + \overline{g}$ be given by $h(z)=\frac{2}{\sqrt{1-z}}$ and $\omega(z)=(g'/h')(z)=z$. Then $f\in\cB_H$ since $h'(z) = (1-z)^{-3/2}$ and
$$
(1-|z|^2) \sqrt{J_f(z)} = \left( \frac{1-|z|^{2}}{|1-z|} \right)^{3/2} \leq 2 \sqrt{2}.
$$
Note that $h\notin\cB$ since, for $0<x<1$, we have
$$
(1-x^2)|h'(x)| = \frac{1+x}{\sqrt{1-x}} \longrightarrow \infty
$$
as $x\to1^-$. Therefore $f$ is not a Bloch function for \cite{Co}. 
\end{example} 

\section{Univalent functions}

Let $f=h+\overline{g}$ be a harmonic, univalent and sense-preserving mapping in $\D$. Let $\omega = g'/h' : \D\to\D$ be its dilatation and write 
$$
h(z)=\sum_{n=0}^\infty a_n z^n \qquad \text{and} \qquad g(z)=\sum_{n=1}^\infty b_n z^n.
$$
We say that $f\in S_H$ if it satisfies $a_0=1-a_1=0$ and that $f\in S_H^0$ if in addition $b_1=0$. 

A simple use of the Schwarz Lemma \cite[\S 5.4]{Du} yields the sharp inequality \mbox{$|b_2|\leq 1/2$} for functions in $S_H^0$. It takes more ef\mbox{}fort to prove that $|a_2|<49$ in $S_H^0$ \cite[\S 6.3]{Du}, and still, the best known constant $49$ is quite distant from the conjectured $5/2$.

For the larger class $S_H$, we have that $|b_1|<1$ simply because $f$ is sense-preserving. Also, it is possible to translate the preceding inequalities by means of an affine transformation. Given $f\in S_H$, the function 
\be \label{aff}
f_0 = \frac{f-\overline{b_1}\overline{f}}{1-|b_1|^2}
\ee
belongs to $S_H^0$. This transformation is invertible, so that $f =f_0 +\overline{b_1}\overline{f_0}$. Hence, it is not difficult to see that 
\be \label{a_2}
|a_2| < 49 + \frac{|b_1|}{2}
\ee
for functions in $S_H$.

In the recent work \cite{HM} a new \emph{Schwarzian} derivative for harmonic locally univalent functions was defined and studied. Also, a \emph{pre-Schwarzian} derivative was defined as 
$$
P_f = \frac{h''}{h'} - \frac{\overline{\omega}\omega'}{1-|\omega|^2}
$$
and with it the following \emph{Becker-type} criterion for univalence was proved.

\begin{theoremO}[\cite{HM}] \label{Becker}
Let $f=h+\overline{g}$ be a sense-preserving harmonic function in the unit disk with dilatation $\omega$. If for all $z\in\D$ 
$$
|zP_f(z)| + \frac{|z\omega'(z)|}{1-|\omega(z)|^2} \leq \frac{1}{1-|z|^2},
$$
then $f$ is univalent.
\end{theoremO}

For any $\omega:\D\to\D$ analytic we define its \emph{hyperbolic derivative} by
$$
\omega^*(z) = \frac{\omega'(z) (1-|z|^2)}{1-|\omega(z)|^2}.
$$
We set $\|\omega\|_h = sup_{z\in\D}|\omega^*(z)|$ for its \emph{hyperbolic norm}. See \cite[\S 5]{BM}.

We recall that a sense-preserving homeomorphism $f$ is called \emph{quasiconformal} if it maps infinitesimal circles onto infinitesimal ellipses having ratio of the major over the minor axis bounded by some constant. This is equivalent to saying that its (second complex) dilatation $\om=\overline{f_{\overline{z}}}/f_z$ is bounded away from one, that is, $|\om(z)| \leq k<1$. See \cite[\S 1.2]{Du}.

\subsection{Connection between univalent and Bloch-type functions.} A well-known theorem of Pommerenke \cite{Po70} gives yet another characterization of the analytic Bloch space $\cB$. It states that a function $f$ is Bloch if and only if there exists a constant $c>0$ and a univalent function $g$ such that $f=c\log g'$. The following theorems show a similar connection between harmonic univalent mappings and the class $\cB_H$. 

\begin{theorem} Let $F=H+\overline{G}$ be univalent and sense-preserving in $\D$. Let \mbox{$h=\log(H')$} and consider any $\omega:\D\to\D$ analytic. Then $f=h+\overline{g}$, having dilatation $\omega_f = \omega$, belongs to $\cB_H$.
\end{theorem}

\begin{proof}
Let $\alpha\in\D$ and compose $F$ with a disk automorphism to obtain 
$$
T(z) = \frac{F\left(\frac{\alpha +z}{1+\overline{\alpha}z}\right) - F(\alpha) }{(1-|\alpha|^2)H'(\alpha)}. 
$$
It can easily be seen that $T\in S_H$ and that the second coefficient of the analytic part of $T$ is given by 
$$
a_2(\alpha) = (1-|\alpha|^2)\frac{H''(\alpha)}{2H'(\alpha)} - \overline{\alpha}.
$$
We turn to $f=h+\overline{g}$ and compute
\begin{align*}
(1-|\alpha|^2)\sqrt{J_f(\alpha)} & \, \leq \, (1-|\alpha|^2) |h'(\alpha)| \\ 
&= \, (1-|\alpha|^2) \left| \frac{H''(\alpha)}{H'(\alpha)} \right| \\
&= \, 2 |a_2(\alpha)+\overline{\alpha}| \\
&< \, 101, 
\end{align*}
in view of \eqref{a_2}. The proof is complete. 
\end{proof}

In the opposite direction we have the following theorem. 
\begin{theorem}
Let $f=h+\overline g \in \cB_H$ be sense-preserving and suppose that $g\in\cB$. Let $0<\ve<1$. Set
$$
H(z) = \int_0^z \exp\left(\frac{\ve}{c} \,h(\ze)\right) d\ze,
$$
where $c = \sqrt{\beta(g)^2 +\beta(f)^2}$, and consider any analytic $\omega:\D\to\D$ satisfying $\|\omega\|_h\leq (1-\ve)/2$. Then $F=H+\overline{G}$, having dilatation $\omega_F = \omega$, is univalent. 
\end{theorem}

\begin{proof}
We apply Theorem \ref{Becker} to the function $F$. Since $f\in\cB_H$ and $g\in\cB$, we have that 
$$
(1-|z|^2)^2 |h'(z)|^2 \, \leq \, \beta(f)^2 +(1-|z|^2)^2 |g'(z)|^2 \, \leq \, c^2. 
$$
Hence
$$
\left| \frac{H''(z)}{H'(z)} \right| \, = \, \frac{\ve}{c} |h'(z)| \, \leq \, \frac{\ve}{1-|z|^2}.
$$
Also, the definition of the hyperbolic norm and our hypothesis lead to
$$
\frac{|\omega'(z)|}{1-|\omega(z)|^2} \, \leq \, \frac{\|\omega\|_h}{1-|z|^2} \, \leq \, \frac{1-\ve}{2(1-|z|^2)}.
$$
We may now compute
\begin{align*}
|zP_F(z)| + \frac{|z\omega'_F(z)|}{1-|\omega_F(z)|^2} \, & \leq \, \left| \frac{H''(z)}{H'(z)} \right| + \frac{2|\omega'(z)|}{1-|\omega(z)|^2} \\
& \leq \, \frac{1}{1-|z|^2}
\end{align*}
and conclude that $F$ is univalent by Theorem \ref{Becker}. 
\end{proof}

\subsection{Schlicht radius.} A well-known covering theorem \cite[\S 6.2]{Du} states that all functions in $S_H^0$ contain in their image a disk centered at the origin, having radius $1/16$. (The conjectured constant is $1/6$.) Applying as before the affine transformation \eqref{aff} it is easy to see that
\be \label{cover}
\left\{ w\in\SC \; : \; |w|<\frac{1-|b_1|}{16}\right\} \subset f(\D)
\ee
for every $f\in S_H$ \cite[Corollary 4.5]{CS}.

A result in the opposite direction states that each function in $S_H$ omits some point on the circle $|w|=\frac{\pi}{2}$. In other words
\be \label{Hall}
\big(\SC\backslash f(\D)\big) \bigcap \left\{|w|=\frac{\pi}{2}\right\} \neq \emptyset.
\ee
The constant $\frac{\pi}{2}$ was given by Hall \cite{Ha} and is best possible. See also \cite[\S 6.2]{Du}.

As mentioned in Section 1, the \emph{schlicht radius} $d_f(z)$ of a harmonic mapping $f=h+\overline{g}$ at a point $z\in\D$ is defined as the radius of the largest disk which is the injective image of some subdomain of $\D$ and is centered at $f(z)$. If there is no such disk then we set $d_f(z)=0$. The existence of a universal lower bound for $\sup_{z\in\D}d_f(z)$ is commonly refered to as a \emph{Bloch theorem}. It was shown in \cite{CGH} that openness (\emph{i.e.}$\!$ the property of mapping open sets to open sets) and the normalization $g'(0)=1-h'(0)=0$ are sufficient conditions for a Bloch theorem to hold. Moreover, it was shown that the normalization alone is not a sufficient condition.

Since here we will be concerned only with univalent functions, the schlicht radius coincides with the distance between $f(z)$ and the boundary of $f(\D)$. The following lemma provides us with some estimates. 

\begin{lemma} \label{radius} 
If $f\in S_H$ then 
$$
\frac{1}{16}(1-|z|^2) (|h'(z)|-|g'(z)|) \; \leq \; d_f(z) \; \leq \; \frac{\pi}{2} (1-|z|^2) |h'(z)|,
$$
for all $z\in\D$. 
\end{lemma}
\begin{proof}
Let $\al\in\D$ and compose with a disk automorphism to obtain 
$$
F(z) = \frac{f\left(\frac{\al +z}{1+\overline{\al}z}\right) - f(\al) }{(1-|\al|^2)h'(\al)} = H(z) + \overline{G(z)}. 
$$
Since $F\in S_H$, the covering theorem \eqref{cover} and Hall's result \eqref{Hall} imply that the radius $d_F(0)$ of the largest disk centered at the origin and contained in the image of $F$ satisfies
$$
\frac{1-|B_1|}{16} \; \leq \; d_{F}(0) \; \leq \; \frac{\pi}{2}.
$$
We compute 
$$
d_{F}(0) = \frac{d_{f}(\al)}{(1-|\al|^2)|h'(\al)|}
$$
and $B_1=g'(\al)/\overline{h'(\al)}$, the first coefficient of $G$. The inequality follows upon substitution. 
\end{proof}

\begin{theorem} \label{rad-B}
Let $f\in S_H$.
\begin{enumerate}
\item[(i)] If $f\in\cB_H$ then $\displaystyle d_f(z) = O\left(\frac{1}{\sqrt{1-|z|}}\right), |z|\to1^-$.

\item[(ii)] If $d_f(z) = O\left(\sqrt{1-|z|}\right), |z|\to1^-$ then $f\in\cB_H$.
\end{enumerate}
If in addition $f$ is quasiconformal then $f\in\cB_H$ if and only if $\,\sup_{z\in\D}d_f(z)<\infty$. 
\end{theorem}

We shall need the following lemma. See \cite{Ga}, page 3. 
\begin{lemmO} \label{L1}
If $\om:\D\to\D$ is analytic then
$$
|\om(z)| \leq \frac{|\om(0)| + |z|}{ 1 +|\om(0)||z|}.
$$
\end{lemmO}

\begin{proof}[Proof of Theorem \ref{rad-B}]
Note that $f\in\cB_H$ is equivalent to 
$$
(1-|z|^2) |h'(z)| \sqrt{1-|\om(z)|^2} \leq \beta(f), \quad z\in\D.
$$
Also note that $\om(0)=b_1$. An application of lemmas \ref{radius} and \ref{L1} yields
$$
d_f(z) \; \leq \; \frac{\pi}{2} (1-|z|^2) |h'(z)| \; \leq \; \frac{\pi}{2} \frac{\beta(f)}{\sqrt{1-|\om(z)|^2}}  \; \leq \; \frac{\pi}{2} \sqrt{\frac{1+|b_1|}{1-|b_1|}} \frac{\beta(f)}{\sqrt{1-|z|}}, 
$$
so that claim (i) is proved. For assertion (ii) we use again lemmas \ref{radius} and \ref{L1} to get
$$
(1-|z|^2)\sqrt{J_f(z)} \; \leq \; 16 \, d_f(z) \sqrt{\frac{1+|\om(z)|}{1-|\om(z)|}} \; \leq \; 16 \sqrt{2} \sqrt{\frac{1+|b_1|}{1-|b_1|}} \frac{d_f(z)}{\sqrt{1-|z|}},
$$
thus $f\in\cB_H$. 

Suppose now that $f$ is quasiconformal and see that its dilatation $\om=g'/h':\D\to\D$ satisfies 
$$
\|\om\|_\infty \; = \; \sup_{z\in\D} |\om(z)| \; < \; 1.
$$
Arguing as before but using only Lemma \ref{radius} we get 
$$
d_f(z) \; \leq \; \frac{\pi}{2} \frac{\beta(f)}{\sqrt{1-\|\om\|_\infty}}
$$
and in the opposite direction 
$$
(1-|z|^2)\sqrt{J_f(z)} \; \leq \; 16 \, d_f(z) \sqrt{\frac{1+\|\om\|_\infty}{1-\|\om\|_\infty}}.
$$
The proof is complete. 
\end{proof}

\section{Growth and coefficients estimates}
For a harmonic sense-preserving function $f = h + \overline{g}$ with dilatation $\om=g'/h':\D\to\D$, we write the power series 
$$
h(z)=\sum_{n=0}^\infty a_n z^n, \quad g(z)=\sum_{n=1}^\infty b_n z^n \quad \text{and} \quad \om(z)=\sum_{n=0}^\infty c_n z^n. 
$$
Of course $c_0=b_1/a_1$. We will also make use of the notation
$$
M_\infty(r,f) = \max_{|z|=r} |f(z)|.
$$
We now present some growth and coefficients estimates for the class $\cB_H$. Note, however, that these bounds are not uniform throughout $\cB_H$, but rather, to each of its subclasses having prescribed $|c_0|$.

\begin{theorem}\label{T1}
If $f = h + \overline{g} \in \cB_H$ is sense-preserving then
$$
\max\{|h(z)-a_0|, |g(z)| \} \leq \beta(f) \sqrt{\frac{1+|c_0|}{1-|c_0|}} \frac{r}{\sqrt{1-r^2}}, \quad |z|=r.
$$
This estimate is sharp in order of magnitude. 
\end{theorem}

\begin{proof}
Let $|z|=r<1$ and write
$$
h(z)-a_0 \, = \, \int_0^z h'(\ze) d\ze \, = \, z \int_0^1 h'(tz) dt.
$$
We have
$$
|h(z)-a_0| \leq r \int_0^1 |h'(tz)| dt \leq r \int_0^1 \frac{\beta(f)}{(1-r^2t^2) \sqrt{1-|\om(tz)|^2}} dt,
$$
since $f\in\cB_H$. We use Lemma \ref{L1} to get
$$
|h(z)-a_0| \, \leq \, \beta(f) \sqrt{\frac{1+|c_0|}{1-|c_0|}}\, r \int_0^1 \frac{dt }{(1-r^2t^2)^{3/2}}.
$$
We compute the integral
$$
\int_0^1 \frac{dt }{(1-r^2t^2)^{3/2}} \, = \, \frac{1}{\sqrt{1-r^2}}
$$
and thus complete the proof of the desired inequality for the function $h$.

We easily get the same bound for $g$ by computing
$$
|g(z)| \, \leq \, r \int_0^1 |g'(tz)| dt
$$
and using the fact that $|g'|\leq|h'|$.

We now prove the sharpness of the order of magnitude. When $c_0=0$, both inequalities (for functions $h$ and $g$) are optimal in view of example \ref{ex-extr}. Our considerations here will contain this as a special case. We take $f = h +\overline{g}$, for which $h'(z) = (1-z)^{-3/2}$, as in example \ref{ex-extr}, but here we take the dilatation to be a self-map of $\D$ whose image is a horodisk centered at some $t\in[0,1)$, that is, $\omega(z) = (g'/h')(z)= t +(1-t)z$. We see that $f\in\cB_H$ since 
\begin{align*}
(1-|z|^2) \sqrt{J_f(z)} & \, = \, \frac{1-|z|^{2}}{|1-z|^{3/2}}\sqrt{1-|\omega(z)|^2} \\
& \, = \,  \frac{1-|z|^{2}}{|1-z|} \sqrt{ \frac{1-|z|^{2} -2t {\rm Re}\,\big(\overline{z}(1-z)\big) -t^2|1-z|^2}{|1-z|}  }\\
& \, \leq \, 2 \sqrt{2}\sqrt{1+t}.
\end{align*}
The sharpness of the inequality for $h$ is now obvious since $h(z)=\frac{2}{\sqrt{1-z}}$ in our example. 

For the function $g$ of this example we compute
$$
g'(z) = \frac{1}{(1-z)^{-3/2}} -\frac{1-t}{\sqrt{1-z}}.
$$
Integrating we get
$$
g(z)=\frac{2}{\sqrt{1-z}}+2(1-t)\sqrt{1-z},
$$
hence, for every $\ve>0$ we have that
$$
(1-x)^{1/2-\ve} |g(x)| \longrightarrow \infty,
$$
when $x\to1^-$. The proof is complete.
\end{proof}

\pagebreak
\begin{theorem}\label{T2}
If $f = h + \overline{g} \in \cB_H$ is sense-preserving then
$$
|a_1| \leq  \frac{\beta(f)}{\sqrt{1-|c_0|^2}} 
$$
and
$$
\max\{|a_n|, |b_n| \} \leq \beta(f) \left(\frac{e}{3}\right)^{3/2} \sqrt{\frac{1+|c_0|}{1-|c_0|}} \sqrt{n+2}, \quad n\geq2.
$$
\end{theorem}

\begin{proof}
For the first inequality we put $z=0$ in the definition of $\cB_H$ and get
$$
\sqrt{|a_1|^2-|b_1|^2} \, \leq \, \beta(f).
$$

Let $n\geq2$. By Cauchy's formula we have that 
$$
|a_n| \, = \, \frac{|h^{(n)}(0)|}{n!} \, = \, \left| \frac{1}{n\,2\pi i} \int_{|\ze|=r} \frac{h'(\ze)}{\ze^n} d\ze\right| \, \leq \, \frac{M_\infty(r,h')}{n \, r^{n-1}},
$$
for any $r\in(0,1)$. Similarly, and also due to the fact that $f$ is sense-preserving, we have that 
$$
|b_n| \, \leq \, \frac{M_\infty(r,g')}{n \, r^{n-1}} \, \leq \, \frac{M_\infty(r,h')}{n \, r^{n-1}}.
$$
The definition of $\cB_H$ implies that 
$$
\max\{|a_n|, |b_n| \} \, \leq \, \frac{\beta(f)}{n \, r^{n-1}(1-r^2)\sqrt{1-M_\infty^2(r,\om)}}.
$$
Using Lemma \ref{L1} we get that 
\begin{align*}
\max\{|a_n|, |b_n| \}  \, & \leq \, \frac{\beta(f) (1+|c_0|r)}{n \, r^{n-1}(1-r^2)^{3/2} \sqrt{1-|c_0|^2}} \\
& \leq \, \frac{\beta(f)}{n} \sqrt{\frac{1+|c_0|}{1-|c_0|}} \frac{1}{r^{n-1}(1-r^2)^{3/2}}.
\end{align*}
This inequality is true for all $r$ in $(0,1)$. Therefore, in order to minimize the expression on the right hand side we see that $r^{n-1}(1-r^2)^{3/2}$ is maximized for $r=\sqrt{\frac{n-1}{n+2}}$. Making this choice we get
\begin{align*}
\max\{|a_n|, |b_n| \}  \leq & \frac{\beta(f)}{n} \sqrt{\frac{1+|c_0|}{1-|c_0|}} \left(\frac{n+2}{n-1}\right)^{\frac{n-1}{2}} \left(\frac{n+2}{3}\right)^{3/2} \\
= & \frac{\beta(f)}{3\sqrt{3}} \sqrt{\frac{1+|c_0|}{1-|c_0|}} \vp(n) \sqrt{n+2},
\end{align*}
where
$$
\vp(x) = \left[\left( 1+ \frac{3}{x-1}\right)^{\frac{x-1}{3}}\right]^{3/2} \left(1+\frac{2}{x}\right).
$$
Note that $\vp(x)\to e^{3/2}$ when $x\to+\infty$. We will now show that $\vp$ increases to its limit. First note that $\vp(x)>0$ for $x\geq2$. We compute
$$
\log\vp(x) = \frac{x-1}{2} \log\left(\frac{x+2}{x-1}\right) +\log\left(\frac{x+2}{x}\right).
$$
Differentiating we get
$$
\psi(x) := \frac{\vp'(x)}{\vp(x)} = \frac{1}{2} \log\left(\frac{x+2}{x-1}\right) - \frac{3x+4}{2x(x+2)}.
$$
One more differentiation yields
$$
\psi'(x) = -\frac{x^2+8}{2x^2(x+2)^2(x-1)}, 
$$
which for $x\geq2$ obviously satisfies $\psi' <0$. Therefore $\psi$ decreases, so that 
$$
\psi(x) > \lim_{x\to\infty} \psi(x) =0,
$$
hence $\vp' >0 $ and the proof is complete.
\end{proof}


\end{document}